\documentclass[12pt,a4paper]{amsart}

\usepackage{amssymb}

\numberwithin{equation}{section}

\newcommand{\bC}{{\mathbb C}}
\newcommand{\bP}{{\mathbb P}}
\newcommand{\bQ}{{\mathbb Q}}
\newcommand{\bR}{{\mathbb R}}
\newcommand{\bZ}{{\mathbb Z}}
\newcommand{\cD}{{\mathcal D}}

\newcommand{\cO}{{\mathcal O}}

\newtheorem{thm}{Theorem}[section]
\newtheorem{prop}[thm]{Proposition}
\newtheorem{lem}[thm]{Lemma}
\newtheorem{cor}[thm]{Corollary}

\newtheorem{quest}[thm]{Question}

\theoremstyle{remark}
\newtheorem{rem}[thm]{Remark}
\newtheorem{ex}[thm]{Example}

\theoremstyle{definition}
\newtheorem{defn}[thm]{Definition}

\title{Compact K\"ahler manifolds with elliptic homotopy type}

\author{Jaume Amor\'os}

\address{Dept.\ Matem\`atica Aplicada I, Universitat Polit\`ecnica de
  Ca\-ta\-lu\-nya, Diagonal 647, 08028 Barcelona, Spain}

\email{jaume.amoros@upc.edu}

\author{Indranil Biswas}

\address{School of Mathematics, Tata Institute of Fundamental Research,
Homi Bhabha Rd, Bombay 400005, India}

\email{indranil@math.tifr.res.in}



\date{}

\begin{document}

\begin{abstract}
Simply connected compact K\"ahler manifolds of dimension up to three
with elliptic homotopy type 
are characterized in terms of their Hodge diamonds. This is applied
to classify the simply connected K\"ahler surfaces and 
Fano threefolds with elliptic homotopy type.
\end{abstract}

\maketitle

\section{Introduction}

Simply connected manifolds with elliptic homotopy type
(see Definition \ref{defb}) constitute an
extension, from the homotopy theoretic point of view,
of the 1--connected homogeneous manifolds.
A manifold $X$ with elliptic
homotopy type has nice homotopical properties, such as:
\begin{itemize}
\item The loop space $\Omega X$
has Betti numbers $$b_n(\Omega X)\,=\, O(n^r)\, ,$$ with $r\,= \,
\dim \pi_{\rm odd}(X) \otimes_{\mathbb Z} \bR$,
and, after localizing at finitely many primes, becomes homotopically 
equivalent 
to a product of spheres.
\item Sullivan's minimal model for $X$ turns into
a finitely generated commutative differential
graded algebra, simplifying its presentation as an algebraic scheme in 
algebro--geometric
homotopy theory (see \cite{KPT} and references therein). 
\item Due to the elliptic--hyperbolic alternative (see Theorem 
\ref{t:eh}), 
the description of the
homotopy groups of a manifold which is not of elliptic homotopy type 
becomes much more 
complex. 
\end{itemize}

Our objective in this work is to identify the simply connected compact 
K\"ahler manifolds,
of dimension up to three, that have elliptic homotopy type. 

For dimension two, we classify them modulo a well--known open question 
in surface theory.

\begin{thm} \label{t:ell2}
A 1--connected compact complex analytic surface has elliptic homotopy 
type 
if and only if it belongs to the following list:
\begin{enumerate}
\item the complex projective plane $\bP^2_{\bC}$,
\item Hirzebruch surfaces ${\mathbb S}_h\,=\,\bP_{\bP^1_{\bC}}(\cO 
\oplus 
\cO(h))$, 
for $h \,\ge\, 0$, and
\item 1--connected general type surfaces $X$ with $q(X)\,=\,p_g(X)\,=\,0$,
$K_X^2\,=\,8$ and $c_2(X)=4$.
\end{enumerate}
\end{thm}

The surfaces of type (iii) in Theorem \ref{t:ell2} are
called simply connected {\em fake quadrics}.
The open question mentioned earlier is whether simply connected
fake quadrics actually exist (see Remark \ref{rem.FQ}). A
simply connected fake quadric,
if it exists, is homeomorphic to either the quadric
${\mathbb S}_0\,=\, \bP^1_{\bC} \times \bP^1_{\bC}$ or
the Hirzebruch surface ${\mathbb S}_1$ (the blow--up of
$\bP^2_{\bC}$ at a point); see Remark \ref{rem.FQ}. Therefore,
Theorem \ref{t:ell2} has the following corollary.

\begin{cor}\label{cor.int.}
A 1--connected compact K\"ahler surface has 
elliptic homotopy type if and only if it is homeomorphic to either 
$\bP^2_{\bC}$ or to a Hirzebruch surface ${\mathbb 
S}_h\,=\,\bP_{\bP^1_{\bC}}(\cO \oplus \cO(h))$ for some $h\, \geq\, 0$.
\end{cor}

Next we classify the compact K\"ahler threefolds with
elliptic homotopy type in terms of the Hodge diamond.

\begin{thm} \label{t:diamonds3}
A 1--connected compact K\"ahler threefold has elliptic homotopy type
if and only if its Hodge diamond is one of the following:
$$
\text{(a)} \quad
\begin{array}{ccccccc}
   &   &   & 1 &   &   &   \\
   &   & 0 &   & 0 &   &   \\
   & 0 &   & 1 &   & 0 &   \\
 0 &   & 0 &   & 0 &   & 0 \\
   & 0 &   & 1 &   & 0 &   \\
   &   & 0 &   & 0 &   &   \\
   &   &   & 1 &   &   &   \end{array}
\qquad \text{(b)} \quad
\begin{array}{ccccccc}
   &   &   & 1 &   &   &   \\
   &   & 0 &   & 0 &   &   \\
   & 0 &   & 2 &   & 0 &   \\
 0 &   & 0 &   & 0 &   & 0 \\
   & 0 &   & 2 &   & 0 &   \\
   &   & 0 &   & 0 &   &   \\
   &   &   & 1 &   &   &   \end{array}
$$
$$
\text{(c)} \quad
\begin{array}{ccccccc}
   &   &   & 1 &   &   &   \\
   &   & 0 &   & 0 &   &   \\
   & 0 &   & 3 &   & 0 &   \\
 0 &   & 0 &   & 0 &   & 0 \\
   & 0 &   & 3 &   & 0 &   \\
   &   & 0 &   & 0 &   &   \\
   &   &   & 1 &   &   &   \end{array}
$$
\end{thm}

All examples known to the authors
of 1--connected compact K\"ahler threefolds with elliptic
homotopy type are in fact rationally connected. 
In the special case of homogeneous manifolds, Borel and Remmert 
proved that 1--connectedness implies rationality \cite{BR}. 
These considerations, and the birational classification of rationally 
connected threefolds
by Koll\'ar, Miyaoka, Mori in \cite{KMM}, have motivated us to
classify the $1$--connected Fano threefolds that are
of elliptic homotopy type. This classification is carried out in
Proposition \ref{p:fanos} by
applying Theorem \ref{t:diamonds3}.
Most of the $1$--connected Fano threefolds with
elliptic homotopy type are neither homogeneous spaces nor
fibrations over lower dimensional manifolds with elliptic homotopy type.

The known examples also led to the following  generalization of the earlier
mentioned question whether simply connected fake quadrics exist:

\begin{quest} \label{q:rc3}
Are there 1--connected compact K\"ahler manifolds with elliptic 
homotopy type that are not rationally connected?
\end{quest}

\medskip

Our results are proved by applying the Friedlander--Halperin bounds
(recalled in Theorem \ref{t:fh}) and
the related properties of the rational
homotopy of finite CW--complexes with elliptic homotopy type. 
The rich topological structure of compact K\"ahler manifolds
arising from Hodge theory constrains the subclass of
elliptic homotopy types.

The definition of elliptic homotopy type may be extended from simply 
connected spaces
to nilpotent spaces, but the
homotopy properties become more complex in that context.
The only nilpotent compact K\"ahler manifolds known to the authors, up 
to finite \'etale coverings, are of the form $X \times T$, where 
$X$ is 1--connected and 
$T$ is a complex torus. It is easy to see that $X\times T$ has
elliptic homotopy type if and only if $X$ has it.

\medskip
\noindent
\textbf{Acknowledgements.}\,
We are grateful to
Ingrid Bauer, Najmuddin Fakhruddin, Ignasi Mundet i Riera and Rita 
Pardini for useful discussions.
We wish to thank TIFR and CSIC for hospitality
that made this work possible. 

\section{The elliptic--hyperbolic dichotomy for homotopy types} \label{s:fh}

The homotopy groups of simply connected finite CW--complexes
show a marked dichotomy, established in \cite{FHT0}, which we will recall
(for its proof and wider discussion, see also \cite[\S~33]{FHT}).

\begin{thm} \label{t:eh} {\em (The elliptic--hyperbolic dichotomy.)}
Let $X$ be a finite, 1--connected CW--complex $X$. The homotopy groups 
of $X$ satisfy one of the two mutually exclusive properties:
\begin{enumerate}
\item $\sum_{i \ge 2} \dim \pi_i(X) \otimes_{\mathbb Z} \bQ \,<\, 
\infty$.
\item $\sum_{i=2}^{k} \dim \pi_i(X) \otimes_{\mathbb Z} \bQ \,>\, 
C^k$
for all $k$ large enough, where $C\,>\,1$ and it depends only on $X$.
\end{enumerate}
\end{thm}

\begin{defn}\label{defb}
The CW--complex $X$ has {\it elliptic homotopy type} if (i) holds.
The CW--complex $X$ has {\it hyperbolic homotopy type} if (ii) holds.
\end{defn}

If $X$ has elliptic homotopy type, then almost all of its homotopy 
groups are torsion, and the Sullivan's minimal model of $X$ is a 
finitely generated algebra determining the rational homotopy type of the 
space. In contrast, if $X$ has hyperbolic homotopy type then 
the Sullivan's 
minimal model of $X$ is a graded algebra, and the number of 
generators grow exponentially with the degree.

For any field $k$ of characteristic zero, the $k$--homotopy groups of a 
1--connected 
finite CW--complex $X$, and the Sullivan's minimal model encoding the 
$k$--homotopy 
type of $X$, may
be obtained from the $\bQ$--homotopy groups and minimal model by 
the extension
of scalars from $\bQ$ to $k$ (see \cite{Sul}). 
The same property of extension of scalars holds for 
cohomology algebras. So we will choose the coefficient field between
$\bQ$ and $\bR$ according to convenience, and will say 
that
$X$ has {\em elliptic homotopy type} or {\em hyperbolic homotopy type}
without any reference to the base field.

We start by presenting examples 
of manifolds with elliptic homotopy type. The first basic examples are:

\begin{ex} \label{ex:lie}
All 1--connected Lie groups and $H$--spaces of finite type have 
elliptic homotopy type.
\end{ex}

\begin{lem} \label{l:fibration}
Let $X \,\longrightarrow\, B$ be a topologically locally trivial
fibration, with fiber $F$, such that $F$, $X$
and $B$ are all 1--connected finite {\rm CW}--complexes.
If any two of them have elliptic homotopy type, then the third one also
has elliptic homotopy type.
\end{lem}

\begin{proof}
Consider the associated long exact sequence of homotopy groups
$$
\ldots \longrightarrow \pi_d(F) \longrightarrow \pi_d(X) \longrightarrow 
\pi_d(B) 
\longrightarrow \ldots\, .
$$
This exact sequence remains exact after tensoring with $\bQ$. Hence the 
lemma follows.
\end{proof}

This leads to the second set of basic examples.

\begin{ex} \label{ex:homog}
A 1--connected homogeneous manifold $X\,=\,G/H$, 
where $H$
is a closed subgroup of a Lie group $G$, has elliptic homotopy type.
It is not necessary that $G$ and $H$ be 1--connected. Since the action 
of $\pi_1$
on the higher homotopy groups is trivial for Lie groups, the homotopy
exact sequence argument of Lemma \ref{l:fibration} carries through
in this case.
\end{ex}

\begin{ex} \label{ex:bundle}
Examples of compact K\"ahler manifolds with elliptic homotopy type 
provided by Lemma \ref{l:fibration} and 
Example \ref{ex:homog} include:
\begin{enumerate}
\item complex projective spaces $\bP^n_{\bC}=U(n+1)/(U(n) \times S^1)$.
\item complex projective space bundles over a basis $B$ of elliptic 
homotopy type, for instance, Hirzebruch surfaces 
$${\mathbb S}_h\,=\,\bP \left( \cO_{\bP^1_{\bC}} \oplus 
\cO_{\bP^1_{\bC}}(h) \right)\, ,$$
where $h$ is a nonnegative integer.
\end{enumerate}
\end{ex}

We now recall a theorem of Friedlander and Halperin.

\begin{thm} \label{t:fh} {\rm (The Friedlander--Halperin bounds, 
\cite{FH}.)}
Let $X$ be a 1--connected, finite {\rm CW}--complex with elliptic homotopy 
type, 
and let $m$ be the maximal degree $d$ such that $H^d(X; \bQ) \,\neq\, 0$.
Select a basis $\{ x_i \}_{i \in I}$
for the odd--degree homotopy $\pi_{\rm odd}(X) \otimes_{\mathbb Z} 
\bQ$, and also select
a basis $\{ y_j \}_{j \in J}$ for the even--degree homotopy
$\pi_{\rm even}(X) \otimes_{\mathbb Z} \bQ$. Then
\begin{enumerate}
\item $\sum_{i \in I} |x_i| \,\le\, 2m-1$\, ,
\item $\sum_{j \in J} |y_j| \,\le\, m$\, ,
\item $\sum_{i \in I} |x_i| - \sum_{j \in J} (|y_j|-1)\,=\, m$,
\item $\sum_{i \in I} |x_i| - \sum_{j \in J} |y_j|\,\ge\, 0$, and
$e(X) \,\ge\, 0$.
\end{enumerate}
(We have denoted by $|x|$ the degree of each homotopy generator $x$,
while $e(X)$ is the topological Euler characteristic of $X$.)
\end{thm}

The above inequality (iv) was proved in \cite{H}. The rest were 
originally established in
\cite{FH}. See \cite[\S~32]{FHT} for a complete proof of it and related
results.

The real homotopy groups of a manifold may be determined by computing the
Sullivan
minimal model of its commutative differential graded algebra ({\em cdga})
of global smooth differential forms. We note that this computation is 
easier for a closed
K\"ahler manifold because such manifolds are formal
(see \cite{DGMS}),
and their Sullivan minimal model is that of the cohomology algebra.
The Friedlander--Halperin bounds in Theorem \ref{t:fh} and the initial steps
in the computation of the Sullivan's minimal model immediately impose 
some bounds on the Betti numbers of manifolds with elliptic 
homotopy type.

\begin{cor} \label{c:b2}
Let $X$ be a 1--connected finite {\rm CW}--complex of elliptic homotopy 
type.
Then, $b_2(X) \,\le\, m/2$, where $m\,=\, \dim X$.
\end{cor}

\begin{proof}
As $X$ is simply connected, by Hurewicz's theorem, there is an isomorphism
$\pi_2(X) \cong H_2(X;\bZ)$. Therefore, it follows from 
inequality (ii) in Theorem \ref{t:fh} that
$$
2 b_2(X) = \sum_{y_j \in \text{ basis of } \\ 
\pi_2(X)\otimes_{\mathbb Z} 
\bQ} |y_j| 
\,\le\, \sum_{j \in J} |y_j| \,\le\, m
$$
completing the proof.
\end{proof}

Likewise, the following bound on the third Betti number
of any 1--connected finite CW--complex $X$ of elliptic homotopy type
can be deduced from Theorem \ref{t:fh}:
$$
b_3(X)+ \dim \ker \left(S^2 H^2(X;\bQ) \stackrel{\cup}{\longrightarrow}
H^4(X;\bQ) \right)\,\le\, (2 \dim X-1)/3
$$
($S^j$ is the $j$--th symmetric product).
We will prove a sharper bound for $b_3$ of closed symplectic manifolds
with elliptic homotopy type (see Proposition \ref{p:b3}).
For that purpose, we will need another homotopical invariant of CW--complexes.

\begin{defn}
Let $X$ be a connected finite CW--complex. 

The {\em Lusternik--Schnirelmann category}
of $X$, denoted ${\rm cat}(X)$, is the least integer $m$ such that 
$X$ can be covered by $m+1$ open subsets each contractible in $X$.

The {\em rational Lusternik--Schnirelmann category} of $X$, denoted 
${\rm cat}_0(X)$, is the least integer $m$ such that there exists $Y$
rationally homotopy equivalent to $X$ with ${\rm cat}(Y)\,=\,m$.
\end{defn}

Some properties of the Lusternik--Schnirelmann category are listed below
(see Proposition 27.5, Proposition 27.14 and \S~28 in 
\cite{FHT} for their proof). 

\begin{prop} \label{p:propcat}
Let $X$ be a connected finite {\rm CW}--complex.
\begin{enumerate}
\item The inequality ${\rm cat}_0(X) \le {\rm cat}(X)$ holds.
\item For any $r$--connected {\rm CW}--complex $X$ of dimension
$m$ ($r\, \geq\, 0$), the inequality ${\rm cat}(X) 
\,\le\, m/(r+1)$ holds.

\item The inequality $${\rm cup}\mbox{--}{\rm length}\,(X) \,\le\,
{\rm cat}_0(X)$$
holds, where ${\rm cup}\mbox{--}{\rm length}\,(X)$ is the
largest integer $p$ such that there exists a product $\alpha_1 \cup 
\ldots \cup \alpha_p \,\neq\, 0$ with $\alpha_i\,\in\,
\oplus_{j>0} H^j(X;\bQ)$.
\end{enumerate}
\end{prop}

Using Proposition \ref{p:propcat}, we get the following property
of symplectic manifolds (an equivalent version of it is
proved in \cite{TO}).

\begin{lem} \label{l:cat}
Let $(X\, , \omega)$ be a 1--connected compact symplectic manifold
with $\dim X\,=\,2n$. Then,
$$
{\rm cat}_0(X)\,=\,{\rm cat}(X)\,=\,n\, .
$$
\end{lem}

\begin{proof}
We may perturb the original symplectic form
$\omega$ to replace it by a symplectic form $\widetilde \omega$
on $X$ arbitrarily close to $\omega$ such that
$[\widetilde \omega] \,\in\, H^2(X; \bQ)$.

The inequalities in Proposition \ref{p:propcat} yield
$$
n \,\le\, {\rm cup}\mbox{--}{\rm length}\,(X) \,\le\,
{\rm cat}_0(X) \,\le\, {\rm cat}(X)
\,\le\, \frac{2n}{2}\,=\,n
$$
completing the proof.
\end{proof}

We now recall another property of CW--complexes that is a natural 
continuation of Theorem \ref{t:fh} (see \cite[\S~32]{FHT}):

\begin{prop} \label{p:odd} {\rm (Friedlander--Halperin.)}
If $X$ is a 1--connected finite {\rm CW}--complex with elliptic homotopy 
type, then
$$
\dim \pi_{\rm odd}(X) \otimes_{\mathbb Z} \bQ \,\le\, {\rm 
cat}_0(X)\, .
$$
\end{prop}

The following is an immediate
consequence of Proposition \ref{p:odd} and Lemma \ref{l:cat}.

\begin{cor} \label{c:oddsympl}
Let $X$ be a 1--connected compact symplectic manifold with
elliptic homotopy type. Then $\dim \pi_{\rm odd}(X) \otimes_{\mathbb 
Z} \bQ \,\le\, \frac{\dim X}{2}$.
\end{cor}

Corollary \ref{c:oddsympl} leads to the following bound on $b_3$ of 
K\"ahler and, more generally, symplectic manifolds.

\begin{prop} \label{p:b3}
Let $X$ be a 1--connected compact symplectic manifold of
dimension $2n$, and let $r$ be the dimension of the 
kernel of
the cup product map
$$\cup \,:\, S^2 H^2(X; \bQ) \,\longrightarrow\, H^4(X; \bQ)\, .$$
Assume that $X$ has elliptic homotopy type. Then $b_3(X) \,\le\, 
n-r$.
\end{prop}

\begin{proof}
Lemma \ref{l:cat} implies that ${\rm cat}_0(X)\,=\,{\rm cat}(X)\,=
\,n$. Therefore, by Corollary \ref{c:oddsympl},
\begin{equation}\label{n1}
\dim \pi_3(X) \otimes_{\mathbb Z} \bQ \,\le\, \dim \pi_{\rm odd} 
(X) \otimes_{\mathbb Z} \bQ\,\le\, n\, .
\end{equation}

The second stage in the computation of the Sullivan's minimal model for
$X$ by induction on cohomology degree (see \cite[Ch. IX]{GM}) shows that 
\begin{equation}\label{n2}
\text{Hom}\,(\pi_3(X), \bQ) \,\cong \,H^3(X; \bQ) \oplus {\rm kernel}\,
\left( S^2 H^2(X; \bQ) \,\stackrel{\cup}{\longrightarrow}\, H^4(X; \bQ) 
\right)\, .
\end{equation}
The proposition follows from \eqref{n1} and \eqref{n2}.
\end{proof}

\section{Compact complex surfaces with elliptic homotopy type} 
\label{s:surfaces}

In this section, 1--connected compact complex surfaces with elliptic homotopy
type are investigated.

\noindent {\bf Theorem \ref{t:ell2}.} {\em
A 1--connected compact complex analytic surface has elliptic homotopy 
type 
if and only if it belongs to the following list:
\begin{enumerate}
\item the complex projective plane $\bP^2_{\bC}$,
\item Hirzebruch surfaces ${\mathbb S}_h\,=\,\bP_{\bP^1_{\bC}}(\cO 
\oplus \cO(h))$, where $h \,\ge\, 0$, and
\item 1--connected surfaces of general type $X$ with $q(X)\,=\,p_g(X)\,=\,0$,
$K_X^2\,=\,8$
and $c_2(X)=4$.
\end{enumerate}
}

Before proving the theorem, we make some remarks on its statement.

\begin{rem}\label{rem.FQ}
Projective surfaces $X$ of general type with $q(X)\,=\,p_g(X)\,=\,0$,
$K_X^2\,=\,8$ and $c_2(X)=4$
are commonly referred to as {\em fake quadrics}. Hirzebruch 
asked whether 1--connected fake quadrics exist. This question
remains open. By Freedman's theorem (see \cite[III, \S~2]{Kir}), any
simply connected fake
quadric is either homeomorphic to the Hirzebruch surface ${\mathbb S}_1$ 
or
to the quadric ${\mathbb S}_0\,=\, \bP^1_{\bC} \times \bP^1_{\bC}$.

The bicanonical map $\Phi_{|2K|}$ of a fake quadric must be of
degree 1 or 2. All fake quadrics with $\deg \Phi_{|2K|}\,=\,2$
have been classified by M. Mendes Lopes
and R. Pardini (see \cite{MLP}),
and each one of them has nontrivial fundamental group.
Many fake quadrics
with bicanonical map of degree one have been found by
Bauer, Catanese, Grunewald and Pignatelli \cite{BCGP}. They
are all uniformized by the bidisk, and have infinite fundamental group. 
\end{rem}

Now we will prove Theorem \ref{t:ell2}.

\begin{proof}
All 1--connected complex analytic surfaces admit
K\"ahler metrics (see \cite[Ch. VI.1]{BPV}). Let $X$ be
a $1$--connected  K\"ahler surface.

The projective
plane and the Hirzebruch surfaces have elliptic homotopy type
(see Example \ref{ex:bundle}). As we noted
in Remark \ref{rem.FQ}, any simply connected fake quadric
is homeomorphic to a Hirzebruch surface.
So if they exist, then they will have elliptic homotopy type as well.

To check that there are no other
$1$--connected K\"ahler surfaces with elliptic homotopy
type, consider the
Hodge numbers of any simply connected K\"ahler surface $X$
with elliptic homotopy type. As $X$ is 1--connected, we 
have $H^1(X;{\mathbb Q})\,=\,0$; so $H^3(X;{\mathbb Q})\,=\,0$ by 
Poincar\'e duality. By Corollary
\ref{c:b2} we know that $b_2(X) \,\le\, 2$. As $h^{1,1}(X) \,\ge\, 1$ and
$h^{2,0}(X)\,=\,h^{0,2}(X)$,
the only possibilities for the Hodge diamond of $X$ are
$$
\text{(a)} \quad 
\begin{array}{ccccc}
   &   & 1 &   &   \\
   & 0 &   & 0 &   \\
 0 &   & 1 &   & 0 \\
   & 0 &   & 0 &   \\
   &   & 1 &   &   \end{array}
\qquad \text{(b)} \quad 
\begin{array}{ccccc}
   &   & 1 &   &   \\
   & 0 &   & 0 &   \\
 0 &   & 2 &   & 0 \\
   & 0 &   & 0 &   \\
   &   & 1 &   &   \end{array}
$$
If $X$ has the Hodge diamond (a), then using the condition that $X$
is simply connected, a theorem of Yau 
implies that $X$ is the complex projective plane
(see \cite[\S~5, Theorem 1.1]{BPV}).

Let $X$ be a simply connected K\"ahler surface possessing the
Hodge diamond (b). Therefore,
\begin{equation}\label{isi1}
\chi(\cO_X)\,=\, 1-q(X)+p_g(X)\,=\,1\,~\, \text{~and~}\, ~\,
c_2(X)\,=\,2-4q(X)+b_2(X)\,=\,4\, .
\end{equation}
Also, we have
\begin{equation}\label{isi2}
c_1(X)^2\,=\,12\cdot\chi(\cO_X) -c_2(X)\,=\,8\, .
\end{equation}

Next, we go over the Kodaira--Enriques classification of surfaces,
ordered by the Kodaira dimension $\kappa(X)$.

$\kappa(X)\,=\, -\infty$:\, Simply connected surfaces in this class are 
rational. Since $h^{1,1}(X)\,=\,2$, there is a nonnegative integer
$h$ such that $X$ is isomorphic to the Hirzebruch surface ${\mathbb 
S}_h$. 

{}From the earlier mentioned theorem of Yau (see \cite[\S~5, Theorem 1.1]{BPV})
if follows that 
the only way for a simply connected surface with Hodge diamond (b) 
to be nonminimal
is to be the blow--up of $\bP^2_{\bC}$ at one point. 
Therefore, if
$\kappa(X) \,\ge\, 0$, we need to examine only the minimal surfaces.

$\kappa(X)\,=\,0$:\, The only 1--connected minimal surfaces in this class
are the K3 surfaces, in which case $c_1^2(X)\,=\, 0$, hence
\eqref{isi2} is contradicted.

$\kappa(X)\,=\,1$:\, In this case $X$ is an elliptic fibration, minimal and
1--connected, with $c_1^2(X)\,=\,0$ (see \cite[Proposition IX.2]{Bea}); this
again contradicts \eqref{isi2}.

$\kappa(X)\,=\,2$:\, This is the case of simply connected fake quadrics. 
As explained in Remark \ref{rem.FQ},
it is currently unknown whether they exist. This completes the
proof of the theorem.
\end{proof}

\section{Compact K\"ahler threefolds with elliptic homotopy type} 
\label{s:solids}

Just as in the case of surfaces, we start by finding out the Hodge
diamonds of compact K\"ahler threefolds with elliptic homotopy
type. From Corollary \ref{c:b2} we know that $b_2(X) \,\le\, 3$, 
while Proposition \ref{p:b3} implies that
$b_3(X) \,\le\, 3$. In fact, a stronger statement holds as
shown by the following proposition.

\begin{prop} \label{p:b3nul}
Let $(X\, ,\omega)$ be a compact 1--connected symplectic 
six dimensional
manifold with elliptic homotopy type. Then $b_3(X)\,=\,0$.
\end{prop}

\begin{proof}
Since $X$ is 1--connected, we have $b_1(X)\,=\,0$. By Poincar\'e
duality,
$b_5(X)\,=\,0$, while $b_4(X)\,=\,b_2(X)$. Since the cohomology
classes represented by $\omega$ and $\omega^2$ are nontrivial,
the topological Euler characteristic $e(X)$ of $X$ admits
the bound
$$e(X) \,=\, 2+ 2 b_2(X)- 
b_3(X) \,\ge\, 4 - b_3(X)\, .$$

We noted above that $b_3(X) \,\le\, 3$. Hence $e(X)\,>0\,$. 
Halperin proved that if $X$ has elliptic homotopy type, and
$e(X)\,>\,0$, then all odd Betti numbers of $X$ vanish
(see \cite[p. 175, Theorem 1'(3)]{H}). This completes the
proof of the proposition.
\end{proof}

Proposition \ref{p:b3nul} and Hodge theory immediately yield
the following.

\begin{cor} \label{c:diamonds3}
If $X$ is a 1--connected compact K\"ahler threefold with 
elliptic homotopy type, then the Hodge diamond
of $X$ is one of the following four:
$$
\text{(a)} \quad
\begin{array}{ccccccc}
   &   &   & 1 &   &   &   \\
   &   & 0 &   & 0 &   &   \\
   & 0 &   & 1 &   & 0 &   \\
 0 &   & 0 &   & 0 &   & 0 \\
   & 0 &   & 1 &   & 0 &   \\
   &   & 0 &   & 0 &   &   \\
   &   &   & 1 &   &   &   \end{array}
\qquad \text{(b)} \quad
\begin{array}{ccccccc}
   &   &   & 1 &   &   &   \\
   &   & 0 &   & 0 &   &   \\
   & 0 &   & 2 &   & 0 &   \\
 0 &   & 0 &   & 0 &   & 0 \\
   & 0 &   & 2 &   & 0 &   \\
   &   & 0 &   & 0 &   &   \\
   &   &   & 1 &   &   &   \end{array}
$$
$$
\text{(c)} \quad
\begin{array}{ccccccc}
   &   &   & 1 &   &   &   \\
   &   & 0 &   & 0 &   &   \\
   & 0 &   & 3 &   & 0 &   \\
 0 &   & 0 &   & 0 &   & 0 \\
   & 0 &   & 3 &   & 0 &   \\
   &   & 0 &   & 0 &   &   \\
   &   &   & 1 &   &   &   \end{array}
\qquad \text{(d)} \quad
\begin{array}{ccccccc}
   &   &   & 1 &   &   &   \\
   &   & 0 &   & 0 &   &   \\
   & 1 &   & 1 &   & 1 &   \\
 0 &   & 0 &   & 0 &   & 0 \\
   & 1 &   & 1 &   & 1 &   \\
   &   & 0 &   & 0 &   &   \\
   &   &   & 1 &   &   &   \end{array}
$$
\end{cor}

The last diamond in the above list is ruled out by the
following proposition.

\begin{prop} \label{p:ke}
There does not exist any simply connected compact
K\"ahler threefold possessing the Hodge diamond (d) in
Corollary \ref{c:diamonds3}.
\end{prop}

\begin{proof}
Let $X$ be a 3--fold having the Hodge diamond (d).
Fix a K\"ahler form $\omega$ on $X$. As 
$H^{1,1}(X)$ is generated by $\omega$, we have that $c_1(X)\,=\, 
\lambda \omega
\,\in\, H^2(X;\bR)$ for some $\lambda \,\in\, \bR$.

First assume that $\lambda\,>\, 0$. Therefore, the
anti--canonical line bundle $\det TX\,:=\,
\bigwedge^{\text{top}}TX$ is positive, so $X$ is complex
projective and Fano. But Fano manifolds 
have $h^{2,0}\,= \, 0$, contradicting the Hodge diamond (d).

Assume now that $\lambda \,=\, 0$. Hence $c_1(X) \,=\,0$. 
Since $X$ is also simply connected,
the canonical line bundle $K_X$ is trivial. Hence
$h^{3,0}\,=\, 1$, which contradicts the Hodge diamond (d).

Lastly, assume that $\lambda \,<\, 0$. This implies that
the canonical line bundle $K_X$ is positive. Therefore,
the Miyaoka--Yau inequality says that
$$
\int_X (c_1^2(X)- 3c_2(X)) \omega \,\le\, 0
$$
(see \cite[p. 449, Theorem 1.1]{Mi}, \cite{Y0}).
Substituting $\omega\,=\,\frac{1}{\lambda} c_1(X)$ in the above
inequality,
\begin{equation}\label{ye}
\int_X c_1^3(X) \,\ge\, 3 \int_X c_1(X) c_2(X)\, .
\end{equation}
But from the Hodge diamond (d) and the Hirzebruch--Riemann--Roch
theorem we derive that 
$$2\,=\,\chi (X, \cO_X)\,=\, \int_X \frac{1}{24} c_1(X) c_2(X)\, .$$
Therefore, $3 \int_X c_1(X) c_2(X)\,=\, 144$, while
$\int_X c_1^3(X)\,=\, \frac{1}{\lambda^3} \int_X \omega^3\,<\, 0$.
This contradicts the inequality in \eqref{ye}, and completes the
proof of the proposition.
\end{proof} 

The following proposition, which is a converse to Corollary 
\ref{c:diamonds3} and Proposition \ref{p:ke},
completes the proof of Theorem \ref{t:diamonds3}.

\begin{prop} \label{p:diamonds3}
If a 1--connected compact K\"ahler threefold has the
Hodge diamond (a),\,(b) or (c) in the list given
in Corollary \ref{c:diamonds3}, then it has elliptic 
homotopy type.
\end{prop}

\begin{proof}
For each of the diamonds (a), (b) and (c), we will find a 
presentation for the real cohomology algebra $H^*(X; \bR)$ of 
any compact K\"ahler threefold $X$ realizing the 
diamond in question. From these presentations we will derive 
Sullivan's minimal model and ellipticity by using a Koszul complex.

The diamond (a) is the easiest to study, because we know from
Hodge theory
that if $\omega$ is a K\"ahler form on $X$, then $\omega^k \,\neq\, 0
\,\in\, H^{k,k}(X)$ for
$k\,=\,1\, ,2\, ,3$. Thus the real cohomology algebra of $X$ is
$H^*(X)\,= \,S^*(y)/(y^4) \,\cong\, H^*(\bP^3_{\bC})$. The formality 
of
closed K\"ahler manifolds implies that such $X$ is real
homotopy--equivalent 
to the complex projective space. Hence $X$ has elliptic type as pointed
out in Example \ref{ex:bundle}(i). Its only nontrivial real homotopy 
groups are $\pi_2(X) \otimes_{\mathbb Z} \bR \,\cong\, \bR$ and
$\pi_7(X) \otimes_{\mathbb Z} \bR \,\cong\, \bR$ (see 
\cite[XIII.A]{GM}).  

Next we consider 
diamond (b). Choose a basis $\{y_1\, , y_2\}$ for $H^2(X; \bR)$ such 
that $y_1$ is the class 
of the K\"ahler form on $X$, and $y_2$ is primitive.
By the Hard Lefschetz Theorem, the pair $\{y_1^2\, , y_1y_2\}$ is a 
basis
of $H^4(X;\bR)$; moreover, $y_1^3$ is the generator of $H^6(X;
\bR)$ with positive
orientation, and $y_1^2 y_2\,=\,0 \,\in\, H^6(X;\bR)$. 

Therefore,
$$
y_2^2\,=\,\alpha y_1^2+ \beta y_1 y_2 \,\in\, H^4(X;\bR)
$$
for some scalars $\alpha\, ,\beta \,\in\, \bR$. The class $y_2$ is real
and primitive. From the
signature of the $Q$--pairing (see \cite[Theorem 6.32]{Voi}),
$$
Q(y_2,\bar y_2)\,= \,- \int_X y_1 y_2^2 \,>\, 0
$$
and $y_1 y_2^2\,=\, \alpha y_1^3+ \beta y_1^2 y_2 \,= \,\alpha y_1^3$.
Hence we have $\alpha \,<\,0$.
Rescaling $y_2$, we may further impose the condition
on the selected basis $\{y_1\, ,y_2\}$ of $H^2(X)$ that
$$y_2^2\,=\, -y_1^2 + \beta y_1 y_2$$
with $\beta \,\in\, \bR$.

We have also shown that the cohomology algebra $H^*(X;\bR)$ is generated by
$H^2(X;\bR)$.
In other words, $H^*(X;\bR)$
is a quotient $\bR [y_1,y_2]/\cD$ of the commutative polynomial ring 
generated by  $y_1$ and $y_2$ by an ideal of
relations that we denote $\cD$. We will
prove that the two already identified relations $p_1(y_1,y_2)\,=\,y_1^2 -\beta 
y_1 y_2 +y_2^2$ and $p_2(y_1,y_2)\,=\,y_1^2 y_2$ actually generate 
the ideal $\cD$.

For any nonnegative integer $k$, let
$$
(\bR [y_1,y_2]/ (p_1, p_2))_k\, \subset\,
\bR [y_1,y_2]/ (p_1, p_2)
$$
be the linear subspace spanned by homogeneous polynomials of degree $k$.
Note that $(\bR [y_1,y_2]/ (p_1, p_2))_1$
has basis $\{y_1\, , y_2\}$, and $(\bR [y_1,y_2]/ (p_1, p_2))_2$ has 
basis $\{y_1^2\, , y_1 y_2\}$,
so they are isomorphic to $H^2(X; \bR)$ and $H^4(X;\bR)$ respectively.

We point out now that $(\bR [y_1,y_2]/ (p_1, p_2))_3$ has basis $y_1^3$. 
First note that $$\dim (\bR [y_1, y_2])_3\,=\,4$$ and $(p_1, p_2)_3$ 
is generated by $y_1 p_1\, , y_2 p_1$ and $p_2$. It is readily 
checked
that these three generators and $y_1^3$ form a basis of $\bR [y_1, y_2])_3$.

Now we will show that $(\bR [y_1, y_2])_4\,=\,(p_1, p_2)_4$.
The obvious inclusions are
\begin{alignat*}{1}
y_1^4 &= y_1^2 p_1 - (\beta y_1 +y_2) p_2 \\
y_1^3 y_2 &= y_1 p_2 \\
y_1^2 y_2^2 &= y_2 p_2 \\
y_1 y_2^3 &= y_1 y_2 p_1 - (y_1+ \beta y_2) p_2
\end{alignat*}
and using them we have $y_2^4\,=\,y_2^2 p_1 -y_2 p_2 -
\beta y_1 y_2^3 \,\in\, (p_1, p_2)_4$.

Finally, for degrees $k \,>\, 4$, 
$$
(\bR [y_1, y_2])_k\,=\, (\bR [y_1, y_2])_{k-4} \cdot (\bR [y_1, y_2])_4 
\,=\, (\bR [y_1, y_2])_{k-4} \cdot (p_1, p_2)_4\, .
$$ 
So assigning 
degree 2 to the variables $y_1\, , y_2$, we have an isomorphism of 
graded 
$\bR$--algebras $$\bR[y_1,y_2]/(p_1, p_2) \,\cong\, H^*(X;\bR)\, .$$

The above presentation for the cohomology algebra of $X$ allows us 
to define 
a cdga $M\,=\,S^*(y_1,y_2) \otimes_{\mathbb R} \wedge^*(x_1,x_2)$ 
with degrees
$$
|y_1|\,=\, |y_2|\,=\,2\, , |x_1|\,=\,3\, ,|x_2|\,=\,5\, ,
$$
and boundaries $dy_1\,=\,dy_2\,=\,0$, $dx_1\,=\, p_1(y_1,y_2)$,
$dx_2\,=\,p_2(y_1,y_2)$. It is equipped with a cdga morphism
$$\rho \,:\, M \,\longrightarrow \,H^*(X)$$ that sends $y_1,y_2$ to their 
namesake cohomology classes, and $x_1, x_2$ to zero.

The algebra $M$ is a Sullivan minimal cdga, so if $\rho$ is a 
quasi--isomorphism, then
$M$ is the minimal model of $X$ and yields the real homotopy 
groups of $X$.

To establish the quasi--isomorphism we are seeking, note that $M$ is a 
{\em pure Sullivan algebra} according to
the definition of \cite[\S~32]{FHT}: 
\begin{itemize}
\item it is finitely generated, 
$M=S^* Q \otimes \wedge^* P$, with $Q\,=\, \langle y_1\, , y_2 \rangle$ 
(the even degree generators), and $P\,=\,\langle x_1, x_2 \rangle$
(the odd degree generators) both finite dimensional; 
\item $d(Q)=0$; 
\item $d(P) \subset S^*Q$.
\end{itemize}
A pure Sullivan algebra has a filtration counting the number of odd degree 
generators in every monomial, and the boundary operator $d$ has degree -1 for 
this filtration. In this way, our differential algebra becomes the 
total space of a homological complex
$$
C_\bullet\,=\,S^* Q \otimes_{\mathbb R} \wedge^\bullet P\, ,
$$
given by
$$
S^* Q \otimes_{\mathbb R} \wedge^2 P \stackrel{d}{\longrightarrow} 
S^* Q \otimes_{\mathbb R} P 
\stackrel{d}{\longrightarrow} S^* Q
$$
which is in fact the Koszul complex for the elements $p_1, p_2 \in S^* 
Q \cong \bR [y_1, y_2]$.
The homology of this complex is therefore $H_\bullet(C_\bullet) 
\,\cong\, 
H^* (M)$; this isomorphism is not graded. There is a short exact 
sequence
$$
0 \longrightarrow H_{>0} (C_\bullet) \longrightarrow 
H_\bullet(C_\bullet) \longrightarrow 
H^*(X) \longrightarrow 0
$$
with the last term being given by isomorphisms $H_0(C_\bullet)
\,\cong\, S^* Q/(p_1, p_2) \,\cong\, H^*(X)$.

It remains now to show that $H_{>0}(C_\bullet)\,=\,0$. To prove this, we 
first
note that the ideal $$\cD\,=\,(p_1,p_2)$$ has
radical $(y_1,y_2)$ in $S^* Q \,=\, \bR [y_1, y_2]$
because of the inclusion of the fourth power $(y_1\, , y_2)^4 
\,\subset\, (p_1, p_2)$. Whenever the ideal 
generated by
$m$ homogeneous elements $p_1, \ldots, p_m$ in a polynomial ring $k 
[y_1, \ldots ,y_m]$
has radical $(y_1, \ldots, y_m)$, the elements $(p_1, \ldots, p_m)$ form 
a
regular sequence of maximal length, and they yield an acyclic Koszul 
complex. 
The version of this property in our cdga setting is:

\begin{prop} \label{p:regular}{\rm \cite[\S~32.3]{FHT}.}
For a pure Sullivan algebra $M\,=\,S^* Q \otimes \wedge^* P$ and 
associated homological complex $C_{\bullet}$ as above,
$$
\oplus_{j>0} H_j (C_\bullet)\,=\,0
$$
if and only if the generators $p_1, 
\ldots, p_m$ of 
$\cD$ form a regular sequence.
\end{prop}

We conclude that $M$ is indeed the minimal model of the 3--fold $X$, 
and its nontrivial
real homotopy groups are  $\pi_2(X) \otimes \bR \,\cong\, \bR^2$,
$\pi_3(X) \otimes \bR
\,\cong\, \bR$ and $\pi_5(X) \otimes \bR \,\cong\, \bR$.

The proof for Hodge diamond (c) employs the same argument used for (b), 
with one additional parameter:

Using the Hard Lefschetz decomposition on real cohomology 
$H^*(X)\,=\,H^*(X; \bR)$,
choose a basis $\{y_1\, , y_2\, , y_3\}$ of $H^2(X)$ such that $y_1$ is 
the class of the
K\"ahler form while $y_2$ and $y_3$ are primitive classes, meaning 
$y_1^2 y_2 \,=\,
0\,=\, y_1^2 y_3$. The elements $\{y_1^2\, , y_1y_2\, , y_1y_3\}$ form a 
basis of $H^4(X)$, and $y_1^3$ generates $H^6(X)$. 

As in the case of diamond (b), the algebra $H^*(X)$ is generated by 
$H^2(X)$, and hence $H^*(X)$
is a quotient of the polynomial ring $S^* H^2(X) \,\cong\, \bR [y_1, 
y_2, y_3]$
by an ideal $\cD$ of boundaries.

The above basis of $H^4(X)$ yields quadratic elements in $\cD$:
\begin{alignat*}{1}
y_2^2 &- \alpha_{11} y_1^2 -\alpha_{12}y_1 y_2 - \alpha_{13} y_1 y_3 \\
y_2 y_3 &- \alpha_{21} y_1^2 -\alpha_{22}y_1 y_2 - \alpha_{23} y_1 y_3 \\
y_3^2 &- \alpha_{31} y_1^2 -\alpha_{32}y_1 y_2 - \alpha_{33} y_1 y_3 
\end{alignat*}
for some $\alpha_{ij}\,\in \,\bR$. As we pointed out
for diamond (b), the $Q$--pairing is symmetric, and it is negative 
definite on real 
primitive 
cohomology, so by a $\bR$--linear change of basis on the primitive cohomology 
group $P^2(X)\,=\, \langle y_2\, , y_3 \rangle$ we obtain a simplified form of 
the above boundaries in $\cD$:
\begin{alignat*}{2}
p_1(y_1,y_2,y_3) &= y_2^2 + y_1^2 &-\beta_{12}y_1 y_2 - \beta_{13} y_1 y_3 \\
p_2(y_1,y_2,y_3) &=y_2 y_3 &-\beta_{22}y_1 y_2 - \beta_{23} y_1 y_3 \\
p_3(y_1,y_2,y_3) &=y_3^2 + y_1^2 &-\beta_{32}y_1 y_2 - \beta_{33} y_1 y_3 
\end{alignat*}
The cohomology algebra $H^*(X)$ satisfies the condition that 
$$y_1^2y_2\,=\, 0
\,=\, y_1^2y_3\, ,$$ and hence $y_1^2y_2$ and $y_1^2y_3$
belong to $\cD$. This fact yields the final simplification among 
the parameters.

\begin{lem} \label{l:2-rels}
With the selected basis for primitive cohomology $P^2(X)$, the set
of quadratic boundaries $\cD_2$ is spanned by the three elements
\begin{alignat*}{2}
p_1(y_1,y_2,y_3) &= y_2^2 + y_1^2 &-\alpha y_1 y_2 - \beta y_1 y_3 \\
p_2(y_1,y_2,y_3) &=y_2 y_3 &-\beta y_1 y_2 - \gamma y_1 y_3 \\
p_3(y_1,y_2,y_3) &=y_3^2 + y_1^2 &-\gamma y_1 y_2 - \delta y_1 y_3 
\end{alignat*}
for some $\alpha, \beta, \gamma, \delta \in \bR$.
\end{lem}

\begin{proof}
First note that
$$
y_1 y_2 y_3\,=\,y_1 p_2 + \beta_{22} y_1^2 y_2 + \beta_{23} y_1^2 y_3 
\,\in\, 
\cD
$$
and then
\begin{alignat*}{2}
y_3 p_1 &= y_1^2 y_3 - \beta_{12} y_1 y_2 y_3 - \beta_{13} y_1 y_3^2 
&\Longrightarrow y_2^2y_3 = \beta_{13} y_1 y_3^2 \mod \cD \\
y_2 p_2 &= y_2^2 y_3 -\beta_{22} y_1 y_2^2 -\beta_{23} y_1 y_2 y_3 
&\Longrightarrow y_2^2 y_3= \beta_{22} y_1 y_2^2 \mod \cD \\
y_3 p_2 &= y_2 y_3^2 -\beta_{22} y_1 y_2 y_3 -\beta_{23} y_1 y_3^2
&\Longrightarrow y_2 y_3^2= \beta_{23} y_1 y_3^2 \mod \cD \\
y_2 p_3 &= y_1^2 y_2 - \beta_{32} y_1 y_2^2 -\beta_{33} y_1 y_2 y_3 + y_2 y_3^2
&\Longrightarrow y_2 y_3^2 = \beta_{32} y_1 y_2^2 \mod \cD
\end{alignat*}
The basis $\{y_2, y_3\}$ of $P^2(X)$ has been so selected that 
the $Q$--pairing yields
$$
y_1 y_2^2\,=\,-y_1^3 \,=\, y_1 y_3^2\, .
$$
This element
is a generator of $H^6(X)$. Therefore the above identities for
$y_2^2 y_3, y_2 y_3^2$ imply that $\beta_{13}\,=\,\beta_{22}$ and
$\beta_{23}\,=\,\beta_{32}$.

In this way we have found three boundary elements that are quadratic 
in $y_1, y_2, y_3$, and are $\bR$--linearly independent. As the
dimension of $H^4(X)$ is three, and
it is generated by products of degree two classes, we conclude that
$\{p_1\, , p_2\, , p_3\}$ is a basis of $\cD_2$. 
\end{proof}

Continuing with the proof of Proposition \ref{p:diamonds3},
we will check that the ideal 
$$
\widetilde \cD\,=\,(p_1, p_2, p_3) \subseteq \cD
$$
is in fact the boundary ideal $\cD$.

It follows from Lemma \ref{l:2-rels} that 
$H^{2k}(X) \,\cong\, \bR [y_1, y_2, y_3]_k/\widetilde \cD_k$ for 
$k\,=\,1,2$.
It also follows from the identities among parameters in $p_1, p_2, p_3$ that
$$y_1^2 y_2, y_1^2 y_3, y_1 y_2 y_3 \,\in\, \widetilde 
\cD_3\, .$$
The choice of $y_2$ and $y_3$ as $Q$--orthogonal means that
$$
y_1 y_2^2\,=\,y_1 y_3^2\,=\, -y_1^3\, \in\, \bR [y_1,
y_2, y_3]/\widetilde \cD\, .
$$
Therefore, $y_1^3$ generates
$\bR [y_1, y_2, y_3]_3/\widetilde \cD_3$, and it is nontrivial in 
cohomology, so it is a generator of the quotient.

Consider $k\,=\,4$. Since any element of
$\bR [y_1, y_2, y_3]_3$ is congruent, modulo $\widetilde \cD$, to
$\lambda y_1^3$ for some $\lambda \,\in\, \bR$, it follows that
any element of $\bR [y_1, y_2, y_3]_4$
is congruent modulo $\widetilde \cD$ to a linear combination of $y_1^4$,
$y_1^3 y_2$ and $y_1^3 y_3$.
The last two monomials are multiples of $y_1^2 y_2$ and $y_1^2 y_3$ 
respectively, and so they lie in $\widetilde \cD$. It is straight 
forward to verify that  $y_1^4$ is a 
linear combination of $y_1^2 p_1$ and elements of $\widetilde 
\cD_4$. The conclusion is that $\bR [y_1, y_2, y_3]_4 \,=\, \widetilde 
\cD_4 \,=\, \cD_4$.

As was observed for diamond (b), for $k>4$,
$$\bR [y_1, y_2, y_3]_k \,=\, 
\bR [y_1, y_2, y_3]_{k-4} \cdot \bR [y_1, y_2, y_3]_4 \,=\, \bR [y_1, 
y_2, y_3]_{k-4} \cdot \widetilde \cD_4 \,=\, \widetilde \cD_k\, .$$ 

Putting all the homogeneous components together, we have proved that 
$\cD=\widetilde \cD$. In other words,
$H^*(X)\,=\, \bR [y_1, y_2, y_3]/(p_1, p_2, p_3)$.

Let $M\,=S^* V^2 \otimes \wedge^* V^3$ be the free cdga
generated
by $V^2\,=\, \langle y_1, y_2, y_3 \rangle$ in degree two and 
$V^3\,=\, \langle x_1, x_2, x_3 
\rangle$ in degree three, with boundaries
$$dy_1\,=\,dy_2\,=\,dy_3\,=0\, ~\,~\, \text{~and~}\,~\,~\, 
dx_j\,=\,p_j\, ~\,~\,~ j\,=\, 1,2,3\,  ,
$$
where $p_j \,\in\, \bR [y_1, y_2, y_3] \,\cong\, S^* V^2$ are the 
generators
of $\cD$.

By its definition, the algebra $M$ is endowed with a morphism 
$$ \rho \,:\, M\longrightarrow H^*(X)\, .
$$
We will prove that $\rho$ is a quasi-isomorphism.

This is done as in the case of Hodge diamond (b): again $M$ is a pure 
Sullivan algebra;
the filtration by number of odd degree variables defines a Koszul complex
structure on $M$, and there is a short exact sequence
$$
0 \longrightarrow H_{>0} (C_\bullet) \longrightarrow H^* M 
\longrightarrow H^*(X) \longrightarrow 0\, .
$$
As before, the radical of $\cD$ is $(y_1, y_2, y_3)$ because of the 
inclusion of its fourth power $(y_1, y_2, 
y_3)^4 \subset
\cD$, thus the 3 generators $p_1, p_2, p_3$ form a regular sequence in
$\bR [y_1, y_2, y_3]$. Thus the Koszul complex $C_\bullet$ has $H_{>0} 
(C_\bullet)\,=\,0$, and $M$ is the minimal model of $H^*(X)$. The only 
nontrivial 
real homotopy groups of $X$ are $\pi_2(X) \otimes \bR\,\cong\,\bR^3$ and
$\pi_3(X) \otimes \bR
\,\cong\, \bR^3$. This completes the proof of Proposition 
\ref{p:diamonds3}.
\end{proof}

Corollary \ref{c:diamonds3}, Proposition \ref{p:ke} and Proposition 
\ref{p:diamonds3} together complete the proof of Theorem \ref{t:diamonds3}.
Examples of $1$--connected projective threefolds realizing the three 
possible Hodge diamonds are
readily found using Lemma \ref{l:fibration}:

\begin{ex} \label{ex:3bundles}
\begin{enumerate}
\item The projective space $\bP^3_{\bC}$ and the quadric threefold
$Q_2 \,\subset\, \bP^4_{\bC}$ have Hodge diamond (a).
\item $\bP^1_{\bC}$--bundles over $\bP^2_{\bC}$ and $\bP^2_{\bC}$--bundles 
over $\bP^1_{\bC}$ have Hodge diamond (b).
\item $\bP^1_{\bC}$--bundles over a Hirzebruch surface ${\mathbb S}_h$, and
bundles of Hirzebruch surfaces over $\bP^1_{\bC}$ have Hodge diamond (c).
\end{enumerate}
\end{ex}
 
We will now consider Fano threefolds. Fano threefolds are
classified up to deformations. Iskovskih classified
Fano threefolds with $b_2\,=\,1$ \cite{Iz}, \cite{Iz89}, while
Mori and Mukai classified Fano threefolds with $b_2 \, > \, 1$
\cite{MM} (all up to deformations). This allows
us to identify the ones having elliptic homotopy type.

The entries in the table of Fano threefolds at the end of \cite{MM}
are henceforth referred to as ``entries'' without further
clarification.

\begin{prop} \label{p:fanos}
Let $X$ be a Fano threefold with elliptic homotopy type.

If $X$ has Hodge diamond (a) in Theorem \ref{t:diamonds3}, then,
up to deformations, $X$ is one of the following:
\begin{enumerate}
\item $\bP^3_{\bC}$,
\item the quadric $Q \,\subset \,\bP^4_{\bC}$, 
\item the Fano manifold $X_{22} \,\subset\, \bP^{13}_{\bC}$ with genus 
$g\,=\,-\frac{1}{2}K_X^3+1\,=\,12$, defined as the subvariety
of the Grassmannian ${\rm Gr}(3,7)$ parametrizing linear spaces $L^3 
\,\subset\, \bC^7$
that are simultaneously isotropic for three general skew--symmetric
forms on ${\mathbb C}^7$ of maximal
rank (see \cite{Mu89}, \cite{Mu92}). 
\end{enumerate}

If $X$ has Hodge diamond (b) in Theorem \ref{t:diamonds3}, then, 
up to deformations, $X$ is one of the following:
entries 20, 21, 22, 24, 26, 27, 29, 30, 31, 32, 33, 34, 35 and 36
in the list of Fano 3--folds with $b_2(X)\,=\,2$ given in
\cite{MM}.

If $X$ has Hodge diamond (c) in Theorem \ref{t:diamonds3},
then, up to deformations, $X$ is one of the following:
entries 5, 8, 10, 12, 13, 15, 16, 17, 18, 19, 20, 21, 22, 23,
24, 25, 26, 27, 28, 29, 30 and 31 in the list of Fano 
3--folds with $b_2(X)\,=\,3$ in \cite{MM}.
\end{prop}

\begin{proof}
As a Fano manifold $X$ has
no holomorphic forms of positive degree, it satisfies 
$b_2(X)\,=\,h^{1,1}(X)$, and $X$ can have Hodge diamond (a), (b) or (c)
in Theorem \ref{t:diamonds3} if and only if $b_3(X)=0$ and $b_2(X)=1,2$ or 3
respectively.

Therefore it suffices to go over the classification of Fano threefolds and 
check the Betti numbers $b_2$ and $b_3$.

First assume that $b_2(X)\,=\,1$.
We use the classification of Fano 3--folds with $b_2(X)\,=\,1$ in 
\cite{Iz}, \cite{Iz89}.

The cases of index $r\,=\,4\, ,3$ are $\bP^3_{\bC}$ 
and the quadric $Q_2 \,\subset\, \bP^4_{\bC}$.
 
The cases $V_d$ of index $r\,=\, 2$ (for which the classification in 
\cite{Iz} is correct), ordered according to their degree $d$
 are:
\begin{itemize}
\item $V_5$ is a section of the Grassmannian $\text{Gr}\,(2,5)$ by 
three general 
hyperplanes in the Pl\"ucker embedding. By the hyperplane section 
theorem,
$b_3(V_5) \,\ge\, b_3(\text{Gr}\,(2,5) \,=\,1$.
\item $V_4$ is the complete intersection of two quadrics in $\bP^5_{\bC}$. 
The hyperplane section and Riemann--Roch--Hirzebruch theorems yield 
that $b_3(V_4)\,=\,4$.
\item $V_3$ is the cubic threefold in $\bP^5_{\bC}$, and by
Riemann--Roch--Hirzebruch has 
$b_3(V_3)\,=\,10$.
\item $V_2$ is a double cover of $\bP^3_{\bC}$ branched over a smooth surface 
of degree 4. Computing the topological Euler number it follows that 
$b_3(V_3)\,=\,20$.
\item $V_1$ is a smooth hypersurface, of degree six, in the weighted 
projective 
space $\bP\{1,1,1,2,3\}$. Computing the topological Euler number it 
follows that $b_3(V_1)\,=\, 42$.
\end{itemize}

The cases with index $r\,=\,1$ are classified according to the genus $g$ 
of the canonical curve section: for every
$$g \,\in\, \{ 2, \ldots, 10, 12 \}\, ,$$ 
there 
is a family or two of Fano manifolds of degree $2g-2$ embedded in 
$\bP^{g+1}_{\bC}$;
we will denote these Fano manifolds by $X_{2g-2}$. The classification of 
the cases with
$r=1$ in \cite{Iz} was corrected in \cite{Iz89}, incorporating 
the missing cases classified by Gushel.
\begin{itemize}
\item $X_2$ is the double cover of $\bP^3_{\bC}$. ramified over
a smooth surface of degree 6. Computing the topological Euler 
number yields that $b_3(X_2)\,=\,104$.
\item For $g=3$, there are two families.\\
First family: $X_4$ a smooth 
quartic 
hypersurface
in $\bP^4_{\bC}$; now, by Riemann--Roch--Hirzebruch, 
$b_3(X_4)\,=\,60$.\\
Second family: 
$X_4$ is the double cover of a quadric $Q_2 \,\subset\, \bP^4_{\bC}$
branched over
a divisor of degree 8; in this case, Riemann--Roch--Hirzebruch and the Euler
number of the cover yield $b_3(X_4)\,=\,60$.

\item $X_6$ is the complete intersection of a quadric and
a cubic hypersurfaces in $\bP^5_{\bC}$, and by
Riemann--Roch--Hirzebruch, $b_3(X_6)\,=\,40$.
\item $X_8$ is the complete intersection of three quadrics in
$\bP^6_{\bC}$, and by Riemann--Roch--Hirzebruch we have $b_3(X_8) \,=\,28$.

\item The case of $g\,=\,6$, which was completed by Gushel, has two 
families.\\
First family:
$X_{10}$ is the transverse intersection of the Pl\"ucker embedding 
of $\text{Gr}\,(2,5)$ in $\bP^9_{\bC}$ with two hyperplanes and a 
quadric.\\
Second family: $X_{10}$, is the intersection of the cone over $V_5$ 
(defined
above) with a quadric in $\bP^6_{\bC}$. The second family is in fact
a deformation of the first one, so both have the same Betti numbers. The 
hyperplane
section theorem shows that $b_3(X_{10}) \,\ge\, 
b_3(\text{Gr}\,(2,5))\,=\,1$.
\item Mukai shows in \cite{Mu92} that $b_3$ for the 
remaining genera are as follows: $b_3(X_{12})\,=\,14$, 
$b_3(X_{14})\,=\,10$, $b_3(X_{16})\,=\,6$,
$b_3(X_{18})\,=\,4$ and $b_3(X_{22})\,=\,0$.
\end{itemize}

The conclusion is that the only Fano 3--folds with $b_2(X)\,=\,1$ and 
$b_3(X)\,=\,0$ are $\bP^3_{\bC}$, the quadric threefold, and $X_{22}$.
We also note that $X_{22}$
is rational homotopy equivalent to the projective space.

Fano threefolds with $b_2(X) \,>\, 1$ are classified by Mori and Mukai 
in
\cite{MM}, and the table at the end of \cite{MM} lists the value
of $b_3$ for each family. So the remaining cases in the theorem 
may be read directly from that table.
\end{proof}

\begin{rem}
Comparing the classification in Proposition \ref{p:fanos}
with the classification of homogeneous complex manifolds of dimension
three in \cite{W} we conclude that $\bP^3_{\bC}$ and the quadric 
threefold $Q \subset \bP^4_{\bC}$ 
are the only homogeneous spaces among the $1$--connected
Fano threefolds with elliptic homotopy type.
\end{rem}

\end{document}